\documentclass{amsart}
\usepackage{mathrsfs,amssymb,mathrsfs, multirow,setspace}
\usepackage[a4paper, total={7in, 9in}]{geometry}
\usepackage{enumerate}
\theoremstyle{plain}
\usepackage{graphicx}
\newtheorem{theorem}{Theorem}[section]
\newtheorem{lemma}[theorem]{Lemma}

\newtheorem{corollary}[theorem]{Corollary}

\theoremstyle{definition}

\theoremstyle{remark}

\input xy
\xyoption{all}

\begin{document}
	\title[Automorphisms of left Ideal relation graph over full matrix ring]{Automorphisms
of left Ideal relation graph over full matrix ring}
	\author[Jitender Kumar, Barkha Baloda, Sanjeet Malhotra]{$\text{Jitender Kumar}^{^*}$, Barkha Baloda, Sanjeet Malhotra}
	\address{Department of Mathematics, Birla Institute of Technology and Science Pilani, Pilani, India}
	\email{jitenderarora09@gmail.com, barkha0026@gmail.com, thesanjeetmalhotra@gmail.com}

	\begin{abstract}
		 The left-ideal relation graph on a ring $R$, denoted by $\overrightarrow{\Gamma_{l-i}}(R)$, is a directed graph whose vertex set is all the elements of $R$ and there is a directed edge from $x$ to a distinct $y$ if and only if the left ideal generated by $x$, written as $[x]$, is properly contained in the left ideal generated by $y$. In this paper, the automorphisms of $\overrightarrow{\Gamma_{l-i}}(R)$ are characterized, where $R$ is the ring of all $n \times n$ matrices over a finite field $F_q$. The undirected left relation graph, denoted by $\Gamma_{l-i}(M_n(F_q))$, is the simple graph whose vertices are all the elements of $R$ and two distinct vertices $x, y$ are adjacent if and only if either $[x] \subset [y]$ or $[y] \subset [x]$ is considered. Various graph theoretic properties of $\Gamma_{l-i}(M_n(F_q))$  including connectedness, girth, clique number, etc.  are studied. 
		 \end{abstract}

	
	\keywords{Matrix ring, left ideals, graph automorphism, finite field\\ *  Corresponding author}
	
	\maketitle

\section{Introduction}
The study of algebraic structures through graph theoretic properties has emerged as a fascinating and important research topic in the past three decades. A lot of graphs associated with rings such as zero divisor graph, co-maximal graph, commuting graph, inclusion ideal graph, and intersection ideal graph  have been studied ( see \cite{akbari2004commuting, akbari2015inclusion, akbari2006zero, akbari2013some, anderson1999zero, beck1988coloring, redmond2002zero, sharma1995note, wang2009co}). Sharma and Bhatwadekar \cite{sharma1995note}, introduced the graph $\Omega(R)$, for a ring $R$, whose vertices are all the elements of $R$ and two distinct vertices $x$ and $y$ are adjacent if and only if $Rx + Ry = R$. They showed that for a commutative ring $R$, the graph $\Omega(R)$ is finitely colorable if and only if $R$ is a finite ring. Miamani \emph{et al.} \cite{maimani2008comaximal}, studied $\Omega_2(R)$ a  subgraph of $\Omega(R)$ whose vertex set is non-unit elements of $R$ and discussed its connectedness and the diameter.  
Meng Ye \emph{et al.} \cite{ye2012co} modified the definition of $\Omega(R)$, and instead of using the elements of $R$ as a vertex set, they used vertices to be the ideals which are not contained in the Jacobson radical of $R$, and two vertices $I_1$ and $I_2$ are adjacent if and only if $I_1 + I_2 = R$. It is referred as the co-maximal ideal graph of $R$. It has a diameter less than or equal to three, and the clique number and the chromatic number of $\Omega(R)$ is equal to the number of maximal ideals of the ring $R$. Inspired by the above mentioned work, and for the idea of revealing relationships between ideals of a ring $R$ and elements of $R$, in \cite{ma2016automorphism}, X. Ma \emph{et al.}, defined a
directed graph, the ideal-relation graph of $R$, written as $\overrightarrow{\Gamma_{i}}(R)$, whose vertex set is $R$ and there is a directed edge
from a vertex $x$ to a distinct vertex $y$ if and only if the ideal of $R$ generated by $x$ is properly
contained in the ideal generated by $y$. 
In this paper, they discussed the automorphisms of $\overrightarrow{\Gamma_{i}}(R)$, where $R$ is the ring of all $n \times n$ upper triangular matrices over a finite field $F_q$. The symmetries of a graph are described by its
automorphism group. In general, automorphism groups are important for studying sizeable graphs since these symmetries allow
one to simplify and understand the behavior of the graph. Although, the determination of the full automorphism group is a challenging problem in algebraic graph theory. Recently, the automorphisms of the zero-divisor graph over the matrix ring attracted the attention of researchers (see \cite{ma2016automorphisms, ou2020automorphism, wang2016automorphisms, wong2014group, zhou2017automorphisms}). Also, Feng Xu \emph{et al.} \cite{xu2020automorphism}, determined all the automorphisms of the intersection graph of ideals over a matrix ring.  Automorphisms of the total graph over matrix rings are characterized in \cite{wang2020automorphisms, zhou2017automorphism}. Further,  D. Wong \emph{et al.} \cite{wang2017automorphisms}, characterized the automorphisms of the co-maximal ideal graph over matrix ring. 

Motivated by the work of \cite{ma2016automorphism, xu2020automorphism}, we consider the left ideal relation graph over full matrix ring. The left-ideal relation graph $\overrightarrow{\Gamma_{l-i}}(R)$ of a ring $R$ is a directed simple graph whose vertex set is all the elements  of $R$ and there is a directed edge from $x$ to a distinct $y$ if and only if the left ideal generated by $x$ is properly contained in the left ideal generated by $y$. The paper is arranged as follows.  In Section 2, we state necessary fundamental notions and fix our notation. Section 3 comprises the results concerning the automorphisms of left-ideal relation of matrix ring  over a finite field $F_q$.  In Section 4, we study various graph invariants of undirected left-ideal relation graph  $\Gamma_{l-i}(R)$ viz.  girth, dominance number, independence number and clique number etc. 

\section{Preliminaries}
In this section, we recall basic definitions from \cite{westgraph} and fix our notation which will be used throughout this paper. Let $\Gamma$ be a graph, $V(\Gamma)$ and $E(\Gamma)$ be the vertex set and edge set, of $\Gamma$ respectively. The \emph{distance} between two vertices $u, v$ in  a graph $\Gamma$ is the number of edges in a shortest path connecting them and it is denoted by $d(u, v)$.  For a vertex $v$, the \emph{eccentricity} of $v$ is the maximum of the distance to any vertex in the graph. The \emph{diameter} is the maximum of the eccentricity of any vertex in $\Gamma$ and the \emph{radius} of $\Gamma$ is the minimum eccentricity among all vertices in the graph $\Gamma$. 
The \emph{degree} of the vertex $v$ in $\Gamma$ is the number of edges incident to $v$ and it is denoted by $deg(v)$. A \emph{cycle} is a closed walk with distinct vertices except for the initial and end vertex, which are equal and a cycle of length $n$ is denoted by $C_n$. The \emph{girth} of $\Gamma$ is the length of its shortest cycle and is denoted by ${g(\Gamma)}$. A connected graph $\Gamma$ is Eulerian if and only if degree of every vertex is even {\cite[Theorem 1.2.26]{westgraph}}. The \emph{chromatic number} of $\Gamma$, denoted by $\chi(\Gamma)$, is the smallest number of colors needed to color the vertices  of $\Gamma$ so that no two adjacent vertices share the same color. A \emph{clique} in $\Gamma$ is a set of pairwise adjacent vertices. The \emph{clique number} of $\Gamma$ is the size of maximum clique in $\Gamma$ and it is denoted by $\omega(\Gamma)$. It is well known that $\omega(\Gamma) \leq \chi(\Gamma)$ (see \cite{westgraph}). 	A subset $D$ of $V(\Gamma)$ is said to be a dominating set if any vertex in $V(\Gamma) \setminus D$ is adjacent to at least one vertex in $D$. If $D$ contains only one vertex then that vertex is called dominating vertex. The \emph{domination number} $\gamma(\Gamma)$ of $\Gamma$ is the minimum size of a dominating set in $\Gamma$. A graph $\Gamma$ is said to be planar if it can be drawn on a plane without any crossing of its edges. For  vertices $u$ and $v$ in a graph $\Gamma$, we say that $z$ \emph{strongly resolves} $u$ and $v$ if there exists a shortest path from $z$ to $u$ containing $v$, or a shortest path from $z$ to $v$ containing $u$. A subset $U$ of $V(\Gamma)$ is a \emph{strong resolving set} of $\Gamma$ if every pair of vertices of $\Gamma$ is strongly resolved by some vertex of $U$. The least cardinality of a strong resolving set of $\Gamma$ is called the \emph{strong metric dimension} of $\Gamma$ and is denoted by $\operatorname{sdim}(\Gamma)$. For  vertices $u$ and $v$ in a graph $\Gamma$, we write $u\equiv v$ if $N[u] = N[v]$. Notice that that $\equiv$ is an equivalence relation on $V(\Gamma)$.
We denote by $\widehat{v}$ the $\equiv$-class containing a vertex $v$ of $\Gamma$.
 Consider a graph $\widehat{\Gamma}$ whose vertex set is the set of all $\equiv$-classes, and vertices $\widehat{u}$ and  $\widehat{v}$ are adjacent if $u$ and $v$ are adjacent in $\Gamma$. This graph is well-defined because in $\Gamma$, $w \sim v$ for all $w \in \widehat{u}$ if and only if $u \sim v$.  We observe that $\widehat{\Gamma}$ is isomorphic to the subgraph $\mathcal{R}_{\Gamma}$ of $\Gamma$ induced by a set of vertices consisting of exactly one element from each $\equiv$-class. Subsequently, we have the following result of \cite{ma2018strong} with $\omega(\mathcal{R}_{\Gamma})$ replaced by $\omega(\widehat{\Gamma})$.

\begin{theorem}[{\cite[Theorem 2.2]{ma2018strong}}]\label{strong-metric-dim}
For any graph $\Gamma$ with diameter $2$,  $\operatorname{sdim}(\Gamma) = |V(\Gamma)| - \omega(\widehat{\Gamma})$.
\end{theorem}
Let $F_q$ be a finite field, where $q$ is prime power, and let $M_n(F_q)$ be the ring of all $n \times n$ matrices over $F_q$. The set of all invertible matrices over $F_q$ will be denoted by $M_n^*(F_q)$. Let $E$ be an $ n \times n$ identity matrix, $E_{s,t}$ the $n \times n$ matrix with  entry in $s^{\text{th}}$ row, $t^{\text{th}}$ column as $1$ rest all are $0$. $J_r$ the matrix $\sum_{i=0}^{r}E_{i,i}$ in $M_n(F_q)$,
$E(i,j)$ the matrix obtained by interchanging $i^{\text{th}}$ and $j^{\text{th}}$ column of identity matrix $E$, $E(i(a))$ the matrix obtained by multiplying $i^{\text{th}}$ column of identity matrix $E$ by $a \in F_q$. Let $I_X$ be the left ideal generated by $X$, for convenience,  we denote it by $[X]$, where $X \in M_n(F_q)$. By $S_X$, we mean the subspace of $F_q^n$ spanned by row vectors of $X \in M_n(F_q)$. Now let $\Gamma$ be a directed graph, $V(\Gamma)$ be the vertex set. For $u, v \in V(\Gamma)$, we write $u \rightarrow v$ if there is a directed edge from $u$ to $v$. Also, by  $ U \sim V$ we mean $u \rightarrow v$ and $v \rightarrow u$. For $u, v \in V(\Gamma)$, we have  $N_i(v) = \{ u ~|\ u \rightarrow v \}$ and  $N_o(v)) = \{ u ~|\ v \rightarrow u \}$.  The in-degree $d_i(v)$ of $v$ is the number of vertices
in  $N_i(v)$. Analogusly, the out-degree $d_o(v)$ of $v$ can be defined as the number of vertices in $N_o(v)$. An automorphism of a graph $\Gamma$ is a permutation $f$ on $V (\Gamma)$ with the property that, for any vertices $u$ and $v$, we
	have $uf \rightarrow vf$ if and only if $u \rightarrow v$. The set $Aut(\Gamma)$ of all graph automorphisms of a graph $\Gamma$ forms a group with
	respect to composition of mappings.

\section{Automorphisms of the left-ideal relation graph of ideals over $M_n(F_q)$}
In this section, we obtain the automorphisms of the left-ideal relation graph of ideals over $M_n(F_q)$. 
Before characterizing all the automorphisms of
$\overrightarrow{\Gamma_{l-i}}(M_n(F_q))$, we introduce two kinds of standard automorphisms for $\overrightarrow{\Gamma_{l-i}}(M_n(F_q))$.

\begin{lemma}\label{varphi}
 For $P \in M_n^*(F_q)$, the map $\varphi_P$ $:\ M_n(F_q) \rightarrow M_n(F_q)$ defined by $\varphi_P(X) = XP $ for each  $X \in M_n(F_q)$, is an automorphism of $\overrightarrow{\Gamma_{l-i}}(M_n(F_q))$.
 \end{lemma}
 \begin{proof}
 Suppose that $X, Y \in M_n(F_q)$ be arbitrary vertices of $\overrightarrow{\Gamma_{l-i}}(M_n(F_q))$ and let $\varphi_P(X) = \varphi_P(Y)$. Then $XP = YP$. Since $P \in M_n^*(F_q)$ we get $X = Y$. Thus, $\varphi_P$ is one-one. Also, $\varphi_P$ is onto as $M_n(F_q)$ is finite. Therefore, $\varphi_P$ is bijective. For $X \rightarrow Y$, we have $[X] \subset [Y]$. Let $Z \in [XP]$ so there exists atleast one $ W \in M_n(F_q)$ such that $Z = WXP$. Thus,  $ZP^{-1} = WX$  so that $ZP^{-1} \in [Y]$. There exists  $V \in M_n(F_q)$ such that $ZP^{-1} = VY$. Consequently, $Z = VYP$ so that $Z \in [YP]$. Therefore, $[XP] \subseteq [YP]$. Let $L \in [Y]$ and $L \notin [X]$ be arbitrary vertices of $M_n(F_q)$. Then $L = UY$ for some $U \in M_n(F_q)$. Consequently, $LP  = UYP$ implies that $LP \in [YP]$. Now assume that $LP \in [XP]$. Then there exists $N \in M_n(F_q)$ such that  $LP = NXP$. It follows that $L = NX$ and so $L \in [X]$, a contradiction. Thus, $[XP] \subset [YP]$ and $\varphi_P(X) \rightarrow \varphi_P(Y)$. Thus, $\varphi_P$ is an automorphism of $\overrightarrow{\Gamma_{l-i}}(M_n(F_q))$.
 \end{proof}
Define $[X]P = \{AP |\ A \in [X]\}$ and for convenience we write as $[X]P = I_{X}P$. As a consequence, $[\varphi_P(X)] = I_{XP}$. 

\begin{lemma}
 For $P\in M_n^*(F_q)$, we have $I_{XP} = I_XP$.
 \end{lemma}
 \begin{proof}
 Let $A \in I_{XP}$ be an arbitrary element of $M_n(F_q)$. Then there exists $C \in M_n(F_q)$ such that $A = CXP$. It follows that $AP^{-1} \in I_X$. Consequently, $A \in I_XP$ implies that $I_{XP} \subseteq I_XP$. Now assume that $B \in I_XP$. Then there exists $D \in M_n(F_q)$ such that $B = DXP$. Therefore, $B \in I_{XP}$. Thus, $I_{XP} = I_XP$.
 \end{proof}
 The next two lemmas are useful in the sequel. 
\begin{lemma}[{\cite[Lemma 2.2]{wang2017automorphisms}}]\label{leftidealgeneratedbysingle}
 Let $I$ be any left ideal of $M_n(F_q)$. Then there exists $X \in M_n(F_q)$ such that $I = [X]$. 
\end{lemma}
\begin{lemma}[{\cite[Lemma 2.2]{wang2017automorphisms}}]\label{lemma4}
 Let $I$ be any left ideal of $M_n(F_q)$. Then there exists $P \in M_n^*(F_q)$ such that $I = I_{E_rP}$, where $0 \leq r \leq n$ and $E_0 = 0$.
\end{lemma}
\begin{lemma}[{\cite[Lemma 3.3]{xu2020automorphism}}]\label{lemma5}
Let $X, Y \in M_n(F_q)$. Then $S_X = S_Y$ if and only if $I_X = I_Y$.
\end{lemma}
\begin{lemma}
Let $X, Y \in M_n(F_q)$. Then $S_X \subset S_Y$ if and only if $I_X \subset I_Y$.
\end{lemma}
\begin{proof}
First suppose that $S_X \subset S_Y$ and let $A \in M_n(F_q)$ be an arbitrary element. Then $A\begin{pmatrix}
  X_1\\ 
  X_2\\
   :\\
  X_n
\end{pmatrix} = \begin{pmatrix}
  X_1^{'}\\ 
  X_2^{'}\\
   :\\
  X_n^{'}
\end{pmatrix}$, where $X_i \in F_q^n$. It follows that $X_i \in S_X$. Since $S_X \subset S_Y$ so that $X_i \in S_Y$. Therefore, there exists $B \in M_n(F_q)$ such that $BY =    \begin{pmatrix}
  X_1^{'}\\ 
  X_2^{'}\\
   :\\
  X_n^{'}
  \end{pmatrix}$. This implies $BY = AX$. Consequently, $AX \in   I_Y$ follows that 
$I_X \subseteq I_Y$. If $I_X = I_Y$, then by Lemma \ref{lemma5} $S_X = S_Y$, a contradiction. Thus, $I_X \subset I_Y$.
Conversely, suppose that $I_X \subset I_Y$ and let $a \in S_X$ such that $a = \sum_{i=1}^{n}c_iX_i$ where $X_i$ is the $i^{th}$ row of matrix $X$ and $c_i \in F_q$. It follows that \begin{center}
$[c_1 c_2 .... c_n]X = a$
\end{center}
Since $I_X \subset I_Y$ so $X \in I_Y$. It implies that $ \begin{pmatrix}
  c\\
  0
\end{pmatrix}X =\begin{pmatrix}
  a\\
  0
\end{pmatrix} \in I_Y$ and for $d \in F_q^n$ we get $\begin{pmatrix}
  d\\
  0
\end{pmatrix}Y = \begin{pmatrix}
  a\\
  0
\end{pmatrix}$. Therefore, $a \in S_Y$.
Thus, $S_X \subseteq S_Y$. If $S_X = S_Y$ then by Lemma \ref{lemma5} $I_X = I_Y$, again a contradiction.
Hence, $S_X \subset S_Y$.
\end{proof}

Let $\Omega$ be the set of $r$ linearly independent $n-\text{dimensional}$ vectors belongs to $F_q^n$ and let $E_\Omega$ be the matrix such that it's first $|\Omega|$ rows are from set $\Omega$ and rest all rows are $0$ vectors. Let us denote $I_\Omega$ as the left ideal generated by $E_\Omega$ and let $I$ be a left ideal of $M_n(F_q)$. By Lemma \ref{leftidealgeneratedbysingle}, we have, $I = I_X$ for some $X \in M_n(F_q)$. Let $\Omega$ be the set of maximal linearly independent row vectors of $X$. Then $S_X = S_{E_\Omega}$ so that $I_X = I_{E_\Omega}$ and  $I = I_\Omega$. From the above discussion, it is easy to observe the following lemma.

\begin{lemma}\label{independentrows}
Let $I$ be a left ideal of $M_n(F_q)$. Then there exists $\Omega \subseteq F_q^n$ such that $\Omega$ is a set of linearly independent row vectors and $I = I_\Omega$.
\end{lemma}

Let $E_{e_i}$ be the matrix such that its first row is $e_i \in F_q^n$ and rest all rows are $0$ for $1 \leq i \leq n$. Let $X \in M_n(F_q)$. Then by Lemma \ref{independentrows}, there exists $\Omega \subseteq F_q^n$ such that $I_X = I_\Omega$. From this we get that $rank(X) = rank(E_\Omega) = |\Omega|$. In the similar manner to $E_{e_i}$, define $E_a$ to be the matrix with its first row as $a \in F_q^n$ and rest all rows are $0$. Clearly  $rank(E_a) = 1$. Let $I$ be a left ideal of $M_n(F_q)$. Then $I$ is a vector space as we can write $aX = (aE)\cdot X \in I$, where $a \in F_q$  and  $X \in I$. Thus, dimension of $I$ is well defined and we claim that $\dim(I) = n\cdot rank(I)$, where  $rank(I) = rank(X)$ such that $I_X = I$. It follows that $I = [E_rP]$, where $P \in M_n^*(F_q)$. Therefore, by Lemma \ref{leftidealgeneratedbysingle}, we have $rank(I) = rank(E_rP) = r$. Notice that  $I_{E_r}$ is a vector space with bases as $\{E_{i,j}\ |\ 1\leq i \leq n\  and \ 1 \leq  j \leq r\}$. Thus, $\dim(I_{E_r}) = nr = n\cdot rank(I)$ and vector space $I \cong I_{E_r}$ as $I = \varphi_P(E_r)$. By Lemma \ref{varphi}, we get that $\varphi_P$ is bijective and one can verify that $\varphi_P$ is a linear transformation. Thus, $\dim(I) = \dim(I_{E_r}) = n\cdot rank(I)$.

\begin{lemma}
 Let $X,Y \in M_n(F_q)$ such that $rank(X) < rank(Y)$. Then $d_i(X) < d_i(Y)$ and $d_o(X) > d_o(Y)$. Moreover
 \begin{enumerate}
    \item[\rm(i)] $d_i(X) = d_i(Y)$ if and only if $rank(X) = rank(Y)$ 
     \item[\rm(ii)]$d_o(X) = d_o(Y)$ if and only if $rank(X) = rank(Y)$.
 \end{enumerate}
 \end{lemma}
\begin{proof}
\rm(i) If $rank(X) = s$ and $rank(Y) = t$, then $s < t$. Assume that $[X] = [E_sA]$ and $[Y] = [E_tB]$, where $A, B \in M_n^*(F_q)$. It follows that $d_i(X) = d_i(E_sA)$ so that $d_i(X) = d_i(E_s)$ because $\varphi_A(E_s) = E_sA$. Similarly, we have $d_i(Y) = d_i(E_t)$. Next, we need to prove that $d_i(E_s) < d_i(E_t)$. Now  $N_i(E_s) \subset N_i(E_t)$ as $E_s \in N_i(E_t)$ but $E_s \notin N_i(E_s)$. Then $d_i(E_s) < d_i(E_t)$ implies that $d_i(X) < d_i(Y)$. 
Further, assume that $d_i(X) = d_i(Y)$. From the above discussion, if we assume that $rank(X) \neq rank(Y)$, then $ d_i(X) \neq d_i(Y)$ which is a contradiction. Thus, $rank(X) = rank(Y)$. 
Now let $rank(X) = rank(Y) = r$. This implies that $d_i(X) = d_i(E_r) = d_i(Y)$. Thus the result holds. 

\rm(ii) The proof is similar to Part (i).
\end{proof} 

\begin{corollary}\label{rank}
Let $\Psi \in Aut(\overrightarrow{\Gamma_{l-i}}(M_n(F_q)))$ and $X \in M_n(F_q)$. Then $rank(\Psi(X)) = rank(X)$.
 \end{corollary}

If $n = 2$ then $rank(X) = 0, 1, 2$,  where $X \in M_2(F_q)$. Note that 
$rank(X) = 0$ if and only if  $[X] = [0]$. Also, 
$rank(X) = 2$ if and only if  $[X] = M_2(F_q)$.

Let $R_i = \{X\ |\ X \in M_2(F_q) \ and \ rank(X) = i\}$. Thus for $X,Y \in M_2(F_q)$, it is clear that $X\rightarrow Y$ if and only if $rank(X) < rank(Y)$. The same has been stated in the following lemma.

\begin{lemma}\label{rankm2}
For $X,Y \in M_2(F_q)$, $X \rightarrow Y$ if and only if $rank(X) < rank(Y)$.
\end{lemma}

Define a mapping $\rho$ from $M_2(F_q)$ to itself such that it permute vertices in $R_i$ for each $i = 0, 1, 2$. Observe that $N_i(X) = N_i(Y)$ and $N_o(X) = N_o(Y)$ for each $ X,Y \in R_i$. Thus, we are in the shape to derive the following lemma.

\begin{lemma}\label{lemma10}
Let $\rho$ be a mapping as defined above. Then $\rho \in Aut(\overrightarrow{\Gamma_{l-i}}(M_2(F_q)))$.
\end{lemma}
\begin{proof}
Since $\rho$ permute vertices in $R_i$, thus it is bijective. Let $X \rightarrow Y$, where $X,Y \in M_2(F_q)$. Then by Lemma \ref{rankm2}, we have $rank(X) < rank(Y)$ and  $rank(\rho(X)) = rank(X)$ and $rank(\rho(Y)) = rank(Y)$ from Corollary \ref{rank}. Therefore, by Lemma \ref{rankm2}, we have $rank(\rho(X)) < rank(\rho(Y))$ implies that $\rho(X) \rightarrow \rho(Y)$. Thus, $\rho \in Aut(\overrightarrow{\Gamma_{l-i}}(M_2(F_q)))$.
\end{proof}

Hence, in what follows, we shall assume  $n \geq 3$.
\begin{lemma}\label{lemma11}
For $\Psi \in Aut(\overrightarrow{\Gamma_{l-i}}(M_n(F_q)))$,  there exists $P \ \in M_n^*(F_q)$ such that $[(\varphi_P\cdot \Psi)(E_{e_1})] = [E_{e_1}]$.
\end{lemma}
\begin{proof}
We know that $rank(E_{e_1}) = 1$ then $rank(\Psi(E_{e_1})) = 1$. Assume that $\Psi(E_{e_1}) = X$ for some $X \in M_n(F_q)$ such that $[\Psi(E_{e_1})] = [E_a]$, where $0\neq a \in F_q^n$. Let $a_k$ be the first non-zero element in row vector $a$. Define \begin{center}
$P = (E-\sum_{i\neq k}a_k^{-1}a_iE_{k,i})E(k(a_k^{-1}))E(1,k)$.
\end{center} 
Note that $P \in M_n^*(F_q)$. Moreover, $[\Psi(E_{e_1})]P = [E_a]P$ which implies that $[\Psi(E_{e_1})P] = [E_aP]$. Therefore, $[(\varphi_P\cdot \Psi)(E_{e_1})] = \begin{pmatrix}
  e_1\\
  0
\end{pmatrix}$ and hence $[(\varphi_P\cdot \Psi)(E_{e_1})] = [E_{e_1}]$.
\end{proof}

\begin{lemma}\label{lemma12}
Let $\Psi \in Aut(\overrightarrow{\Gamma_{l-i}}(M_n(F_q)))$ such that $[E_{e_k}] = [\Psi(E_{e_k})]$, where $1 \leq k \leq i$. Then $[E_i] = [\Psi(E_i)]$.
\end{lemma}
\begin{proof}
For $i=1$, we have $E_{e_1} = E_1$, we are done. If $i\geq 2$, then  $[E_{e_k}] \subset [E_i]$. It implies that $E_{e_k}\rightarrow E_i$ and so $\Psi(E_{e_k}) \rightarrow \Psi(E_i)$. Thus, $[\Psi(E_{e_k})] \subset [\Psi(E_i)]$ so that $[E_{e_k}] \subset [\Psi(E_i)]$. It follows that $\sum_{k=1}^{i}[E_{e_k}] \subseteq [\Psi(E_i)]$. Therefore, $[E_i] \subseteq [\Psi(E_i)]$. Since $\dim([E_i]) = \dim([\Psi(E_i)]) = n\cdot i$ we get $[E_i] = [\Psi(E_i)]$.
\end{proof}

\begin{lemma}\label{lemma13}
Let $\Psi \in Aut(\overrightarrow{\Gamma_{l-i}}(M_n(F_q)))$ such that $[E_{e_k}] = [\Psi(E_{e_k})]$ and $1 \leq k \leq i-1$. Then there exists $P \in M_n^*(F_q)$ such that $[(\varphi_P\cdot \Psi)(E_{e_s})] = [E_{e_s}]$, where $1 \leq s \leq i$.
\end{lemma}
\begin{proof}
Let $\Psi(E_{e_i}) = X$ for some $X \in M_n(F_q)$ such that $[\Psi(E_{e_i})] = [E_a]$, where $a (\neq 0) \in F_q^n$. Since $[E_{e_i}]$ is not a proper subset of $[E_{i-1}]$ then $[\Psi(E_{e_i})]$ is not a proper subset of $[\Psi(E_{i-1}]$. By Lemma \ref{lemma12}, we have $[E_{i-1}] = [\Psi(E_{i-1})]$. It follows that $[\Psi(E_{e_i})]$ is not a proper subset of $[E_{i-1}]$. Then there exists $ a_l (\neq 0) \in F_q$ such that $i \leq l \leq n$. Now define \begin{center}
$P = (E-\sum_{j\neq l}a_l^{-1}a_jE_{l,j})E(l(a_l^{-1}))E(i,l)$. 
\end{center}
Note that $P \in M_n^*(F_q)$. Moreover, $[\Psi(E_{e_i})]P = [E_a]P$ and  $[\Psi(E_{e_i})P] = [E_aP]$. Therefore,  $[(\varphi_P\cdot \Psi)(E_{e_i})] = \begin{pmatrix}
  e_i\\
  0
\end{pmatrix}$ so that $[(\varphi_P\cdot \Psi)(E_{e_i})] = [E_{e_i}]$. Now, $[\Psi(E_{e_k})]P = [E_{e_k}]P$, where $1 \leq k < i$,  implies that $[\Psi(E_{e_k})P] = [E_{e_k}P]$ and therefore, $[(\varphi_P\cdot \Psi)(E_{e_k})] = [\varphi_P(E_{e_k})]$. It is  easy to observe that $[\varphi_P(E_{e_k})] = [E_{e_k}]$. Thus, $[(\varphi_P\cdot \Psi)(E_{e_k})] = [E_{e_k}]$, where $1 \leq k \leq i-1$. Hence, $[(\varphi_P\cdot \Psi)(E_{e_s})] = [E_{e_s}]$, where $1 \leq s \leq i$.
\end{proof}   

\begin{corollary}\label{corollary14}
Let $\Psi \in Aut(\overrightarrow{\Gamma_{l-i}}(M_n(F_q)))$. Then there exists  $P \in M_n^*(F_q)$ such that $[(\varphi_P\cdot \Psi)(E_{e_k})] = [E_k]$ and $1 \leq k \leq n$.
\end{corollary}
\begin{proof} 
In view of Lemma \ref{lemma11}, there exists $P_1 \in M_n^*(F_q)$ such that $[(\varphi_{P_1}\cdot \Psi)(E_{e_1})] = [E_{e_1}]$. Then by Lemma \ref{lemma13}, we get $P_2 \in M_n^*(F_q)$ such that $[(\varphi_{P_2}\cdot \varphi_{P_1}\cdot \Psi)(E_{e_k})] = [E_{e_k}]$, where $k = 1, 2$. On the similar lines of proof of Lemma \ref{lemma13}, there exists $P_1, P_2, \cdots,  P_n$ such that $[(\varphi_{P_n}\cdot \cdots \varphi_{P_2}\cdot \varphi_{P_1}\cdot \Psi)(E_{e_k})] = [E_{e_k}]$ where $k = 1, 2, \cdots, n$. Assume that $P = P_1P_2 \cdots P_n$ and  $P \in M_n^*(F_q)$ and observe that $\varphi_{Q_2}\cdot \varphi_{Q_1} = \varphi_{Q_1Q_2}$, where $Q_1, Q_2 \in M_n^*(F_q)$. This implies that $\varphi_{P_n}\cdot \cdots \varphi_{P_2}\cdot \varphi_{P_1} = \varphi_{P_1P_2 \cdots P_n}$. Therefore, $[(\varphi_P\cdot \Psi)(E_{e_k})] = [E_{e_k}]$, where $k = 1,2, \cdots, n$.
\end{proof}

Let $\Delta = \{e_1, e_2, \cdots,  e_n\}$ be the set of all unit vectors of $F_q^n$.

\begin{lemma}\label{lemma15}
Let $\delta \subseteq \Delta$. If  $\Psi \in Aut(\overrightarrow{\Gamma_{l-i}}(M_n(F_q)))$ such that $[\Psi(E_{e_k})] = [E_{e_k}]$ and  $k = 1,2, \cdots, n$, then $[\Psi(E_{\delta})] = [E_{\delta}]$.
\end{lemma}
\begin{proof}
If $|\delta| = 1$ then $E_{\delta} = E_{e_k}$. Therefore, we assume that $|\delta|\geq 2$. Let $e_k \in \delta$. Then $[E_{e_k}] \subset [E_{\delta}]$ implies that $[\Psi(E_{e_k})] \subset [\Psi(E_{\delta})]$. Consequently, we obtain $[E_{e_k}] \subset [\Psi(E_{\delta})]$.
Let $I(\delta) = \{i\ |\ e_i \in \delta\}$. Then we have 
    $\sum_{k \in I(\delta)}[E_{e_k}] \subseteq [\Psi(E_{\delta})]$. This implies that $[E_{\delta}] \subseteq [\Psi(E_{\delta})]$. Since $\dim([E_{\delta}]) = \dim([\Psi(E_{\delta})]) = n\cdot |\delta|$ we get, $[E_{\delta}] = [\Psi(E_{\delta})]$.
\end{proof}

For any $X \in M_n(F_q)$ such that $rank(X) = 1$ we have $[X] = [E_a]$, where $a \in F_q^n$. Note that such representation of $[X]$ is not unique and there can be a $b \in F_q^n$ such that $[X] = [E_b]$. But if we place a condition that $a$ should be such that its first non-zero element be $1$ then we get a unique $a \in F_q^n$ such that $[X] = [E_a]$. In next few lemmas we will consider rank $1$ matrices of $M_n(F_q)$.

\begin{lemma}\label{lemma16}
Let $\Psi \in Aut(\overrightarrow{\Gamma_{l-i}}(M_n(F_q)))$ such that $[\Psi(E_{e_k})] = [E_{e_k}]$, where $k = 1,2, \cdots, n$. If $X \in M_n(F_q)$ such that $rank(X) = 1$, $[X] = [E_a]$ and $[\Psi(X)] = [E_b]$, where $a = (a_{1}, a_{2}, \cdots, a_{n}),~ b = (b_{1}, b_{2}, \cdots, b_{n}) \in F_q^n$, then $a_l = 0$ if and only if $b_l = 0$, where $1 \leq l \leq n$.
\end{lemma}
\begin{proof}
Let $a_l = 0$. For each $Y \in [E_a]$, we have $l^{th}$ column of $Y = 0$. Let us denote $\Delta\backslash\{e_l\}$ as $\Delta_l$. Then $[E_a] \subset [E_{\Delta_l}]$ implies that $[\Psi(X)] \subset [\Psi(E_{\Delta_l})]$. Consequently, $[E_b] \subset [\Psi(E_{\Delta_l})]$. By Lemma \ref{lemma15}, we obtain $[E_{\Delta_l}] = [\Psi(E_{\Delta_l})]$. Therefore,  $[E_b] \subset [E_{\Delta_l}]$ which implies that $b_l = 0$.

Further suppose that $a_l \neq 0$. Then $[E_a] \not \subset [E_{\Delta_l}]$. It follows that $[\Psi(X)] \not\subset [\Psi(E_{\Delta_l})]$ and so $[E_b] \not \subset [\Psi(E_{\Delta_l})]$. By Lemma \ref{lemma15}, we get $[E_b] \not \subset [E_{\Delta_l}]$. Now, $b_l = 0$  gives $[E_b] \subset [E_{\Delta_l}]$, a contradiction. Thus, $b_l \neq 0$.
\end{proof}

\begin{lemma}\label{lemma17}
Let $\Psi \in Aut(\overrightarrow{\Gamma_{l-i}}(M_n(F_q)))$ such that $[\Psi(E_{e_k})] = [E_{e_k}]$ where $k = 1,2, \cdots, n$. Let $a_{i} = (a_{i1}, a_{i2}, \cdots, a_{in}) \in F_q^n$ such that $[X_i] = [E_{a_i}]$ where $X_i \in M_n(F_q)$ and $i = 1, 2$. Let $b_i = (b_{i1}, b_{i2}, \cdots, b_{in}) \in F_q^n$ such that $[\Psi(X_i)] = [E_{b_i}]$. Then $b_{1s}b_{2k} = b_{1k}b_{2s}$ if and only if $a_{1s}a_{2k} = a_{1k}a_{2s}$, where $ 1 \leq s < k \leq n$.
\end{lemma}
\begin{proof}
 Suppose that $a_{1s}a_{2k} = a_{1k}a_{2s}$. If $a_{1s}a_{2k}a_{1k}a_{2s} = 0$ then, by Lemma \ref{lemma16}, $b_{1s}b_{2k}b_{1k}b_{2s} = 0$. Therefore, $a_{1s}a_{2k} = a_{1k}a_{2s}$ implies $b_{1s}b_{2k} = b_{1k}b_{2s}$. Thus, we assume $a_{1s}a_{2k}a_{1k}a_{2s} \neq 0$. By Lemma \ref{lemma16}, we have  $b_{1s}b_{2k}b_{1k}b_{2s} \neq 0$. Consider the sets 
\begin{equation*}
\begin{split}
\alpha = \{e_1, e_2, \cdots, a_{1s}e_s+a_{1k}e_k, \cdots, e_{k-1}, e_{k+1}, \cdots, e_n\}, \text{and}\\
\beta = \{e_1, e_2, \cdots, b_{1s}e_s+b_{1k}e_k,\cdots, e_{k-1}, e_{k+1}, \cdots, e_n\}
\end{split}
\end{equation*}
Note that $[E_{a_1}] \subset [E_\alpha]$. It follows that $[\Psi(X_1)] \subset [\Psi(E_\alpha)]$ and so  $[E_{b_1}] \subset [E_\alpha]$. And, $[E_{\Delta\backslash\{e_s, e_k\}}] \subset [E_\alpha]$ so $ [E_{\Delta\backslash\{e_s, e_k\}}] \subset [\Psi(E_\alpha)]$. Consequently, $[E_{\Delta\backslash\{e_s, e_k\}}] \ + \ [E_{b_1}]\subseteq [\Psi(E_\alpha)]$. Since $\dim([E_{\Delta\backslash\{e_s, e_k\}}] \ + \ [E_{b_1}]) = \dim([\Psi(E_\alpha)]) = n(n-1)$, we have  $[E_{\Delta\backslash\{e_s, e_k\}}] \ + \ [E_{b_1}] = [\Psi(E_\alpha)]$. It is easy to observe that $[E_{\Delta\backslash\{e_s, e_k\}}] \ + \ [E_{b_1}] = [E_\beta]$. Thus, $[E_\beta] = [\Psi(E_{\alpha})]$.

If $a_{1s}a_{2k} = a_{1k}a_{2s}$ then $[E_{a_2}] \subset [E_\alpha]$. It follows that $[X_2] \subset [E_\alpha]$ which implies that $[E_{b_2}] \subset [E_\beta]$. Therefore, $b_{1s}b_{2k} = b_{1k}b_{2s}$.

Now assume that $a_{1s}a_{2k} \neq a_{1k}a_{2s}$. It follows that $[E_{a_2}] \not\subset [E_\alpha]$. Consequently, $[X_2] \not\subset [E_\alpha]$ gives $[E_{b_2}] \not\subset [E_\beta]$. If $b_{1s}b_{2k} = b_{1k}b_{2s}$, then we get $[E_{b_2}] \subset [E_\beta]$ a contradiction. Thus, $b_{1s}b_{2k} \neq b_{1k}b_{2s}$.
\end{proof}

\begin{lemma}\label{lemma18}
Let $\Psi \in Aut(\overrightarrow{\Gamma_{l-i}}(M_n(F_q)))$. Then there exists $P \in M_n^*(F_q)$ such that $[(\varphi_P\cdot \Psi)(E_\mathbf{1})] = [E_\mathbf{1}]$ 
where $[E_\mathbf{1}] = \begin{pmatrix}
\mathbf{1}\\
0
\end{pmatrix}$ and $\mathbf{1} = (1,1,  \cdots, 1) \in F_q^n$ i.e. all-one vector. Also, $[(\varphi_P\cdot \Psi)(E_{e_k})] = [E_{e_k}]$, where $1 \leq k \leq n$.
\end{lemma}
\begin{proof}
 In view of Corollary \ref{corollary14}, there exists a matrix $P_1 \in M_n^*(F_q)$ such that $[(\varphi_{P_1}\cdot \Psi)(E_{e_k})]=[E_{e_k}]$, where $1 \leq k \leq n$. Let $[(\varphi_{P_1}\cdot \Psi)(E_\mathbf{1})]  = [E_a]$, where $a \in F_q^n$. Then by Lemma \ref{lemma16}, $a_l \neq 0$ for each $1 \leq l \leq n$. 
Now, define $P_2 \in M_n^*(F_q)$ such that $P_2 = diag(a_1^{-1}, a_2^{-1}, \cdots, a_n^{-1})$. We obtain $[(\varphi_{P_1}\cdot \Psi)(E_\mathbf{1})]P_2\   =\  [E_a]P_2$. It follows that $[(\varphi_{P_1}\cdot \Psi)(E_\mathbf{1})P_2]  = [E_aP_2]$. It implies that $[(\varphi_{P_2}\cdot \varphi_{P_1}\cdot \Psi)(E_\mathbf{1})]  = [E_\mathbf{1}]$ so that $[(\varphi_{P_1P_2}\cdot \Psi)(E_\mathbf{1})] = [E_\mathbf{1}]$. Let $P = P_1P_2 \in M_n^*(F_q)$. Then $[(\varphi_P\cdot \Psi)(E_\mathbf{1})] = [E_\mathbf{1}]$. And, $[(\varphi_{P_1P_2}\cdot \Psi)(E_{e_k})] = [\varphi_{P_2}(E_{e_k})] = [E_{e_k}]$. Consequently, $[(\varphi_P\cdot \Psi)(E_{e_k})] = [E_{e_k}]$, where $ 1 \leq k \leq n$. Thus, $\varphi_P\cdot \Psi$ is the required automorphism.
\end{proof}
\begin{lemma}\label{lemma19}
Let $\Psi \in Aut(\overrightarrow{\Gamma_{l-i}}(M_n(F_q)))$ such that $[\Psi(E_{e_k})] = [E_{e_k}]$, where  $\ 1 \leq k \leq n$ and $[\Psi(E_\mathbf{1})] = [E_\mathbf{1}]$. Suppose $X \in M_n(F_q)$ such that $[X] = [E_a]$ and $[\Psi(X)] = [E_b]$, where $a, b \in F_q^n$. Then $a_s = a_k$ if and only if $b_s = b_k$ and $1 \leq s < k \leq n$.
\end{lemma}
\begin{proof}
Let $a_s = a_k$. Then $[\Psi(E_\mathbf{1})] = [E_\mathbf{1}]$. It follows that $a_s\cdot 1 = 1\cdot a_k$ and so $b_s\cdot 1 = 1\cdot b_k$.  By Lemma \ref{lemma17}, we have $b_s = b_k$. If $b_s = b_k$, then $b_s\cdot 1 = 1\cdot b_k$ gives  $a_s\cdot 1 = 1\cdot a_k$. By Lemma \ref{lemma17}, we get $a_s = a_k$.
\end{proof}

Let $\Psi \in Aut(\overrightarrow{\Gamma_{l-i}}(M_n(F_q)))$ such that $[\Psi(E_{e_k})] = [E_{e_k}]$, where $\ 1 \leq k \leq n$ and $[\Psi(E_\mathbf{1})] = [E_\mathbf{1}]$. Let $X \in M_n(F_q)$ such that $[X] = [E_\alpha]$ where $\alpha = (e_1\ +\ ae_2)$ and $a \in F_q$. Further assume that $[\Psi(X)] = [E_\beta]$ where $\beta = {e_1 \ + \ a'e_2}$ and $a' \in F_q$. Note that $a'$ depends upon $a$. Hence, we can define a mapping $\Upsilon$ on $F_q$ such that $\Upsilon(a) = a'$. Thus, $\Upsilon$ is one-one because if $\Upsilon(u) = \Upsilon(v)$, where $u,v \in F_q$, then by Lemma \ref{lemma17}, we get $u = v$. As $\Upsilon$ is a one-one mapping over finite field $F_q$. It follows that $\Upsilon$ is onto. Thus, $\Upsilon$ is bijective over $F_q$. Moreover, by Lemma \ref{lemma16} and Lemma \ref{lemma17} we get, $\Upsilon(0) = 0$ and $\Upsilon(1) = 1$.

\begin{lemma}\label{lemma20}
Let $\Psi \in Aut(\overrightarrow{\Gamma_{l-i}}(M_n(F_q)))$ such that $[\Psi(E_{e_k})] = [E_{e_k}]$, where $1 \leq k \leq n$ and $[\Psi(E_\mathbf{1})] = [E_\mathbf{1}]$. Let $\Upsilon$ be as defined above. Then the following holds:
\begin{enumerate}
\item[{\rm(i)}] If $[X] = [E_{\{e_1+ae_k\}}]$, where $X \in M_n(F_q)$, then $[\Psi(X)] = [E_{\{e_1+\Upsilon(a)e_k\}}]$ for each $2 \leq k \leq n$ and $a \in F_q$.
\item[{\rm(ii)}] If $[X] = [E_{\{e_1 + a_2e_2 + \cdots + a_n e_n\}}]$, where $X \in M_n(F_q)$, then $[\Psi(X)] =$ $[E_{\{e_1+\Upsilon(a_2)e_2+\cdots+\Upsilon(a_n)e_n\}}]$, where $a_i \in F_q$ and $2\leq i \leq n$.
\item[{\rm(iii)}] If $[X] = [E_{\{e_{i}+a_{i+1}e_{i+1}+ \cdots +a_n e_n\}}]$, where $X \in M_n(F_q)$, then $[\Psi(X)] = $ $[E_{\{e_i +\Upsilon(a_{i+1})e_{i+1}+ \cdots +\Upsilon(a_n)e_n\}}]$, when $ a_j \in F_q$ and $i < j \leq n$.
\end{enumerate}
\end{lemma}
\begin{proof}

\begin{enumerate}
\item[{\rm(i)}] For $k = 2$, there is nothing to prove. We need to prove for $k \geq 3$. Let $Y \in M_n(F_q)$ such that $[Y] = [E_{\{e_1+ae_2+ae_k\}}]$ and $[\Psi(Y)] = [E_{\{e_1+a'e_2+a''e_k\}}]$. By Lemma \ref{lemma19}, we get $a' = a''$ so $[\Psi(Y)] = [E_{\{e_1+a'e_2+a'e_k\}}]$. Let $[Z] = [E_{\{e_1+ae_2\}}]$. Then $[\Psi(Z)] = [E_{\{e_1+\Upsilon(a)e_2\}}]$. By Lemma \ref{lemma17}, and using $[Y]$ and $[Z]$, we get $a' = \Upsilon(a)$. Thus, $ [\Psi(Y)] = [E_{\{e_1+\Upsilon(a)e_2+\Upsilon(a)e_k\}}]$. Let $[\Psi(X)] = [E_{\{e_1+de_k\}}]$. By Lemma \ref{lemma17} and using $[X]$ and $[Y]$, we get $d = \Upsilon(a)$.
\item[{\rm(ii)}] Suppose that $[\Psi(X)] = [E_{\{e_1+c_2e_2+\cdots+c_ne_n\}}]$. Let $Y_k \in M_n(F_q)$ such that $[Y_k] = [E_{\{e_1+a_ke_k\}}]$. It follows that $[\Psi(Y_k)] = [E_{\{e_1+\Upsilon(a_k)e_k\}}]$. By Lemma \ref{lemma17} and by using $[Y_k]$ and $[X]$, we get $c_k = \Upsilon(a_k)$ for each $2 \leq k \leq n$.

\item[{\rm(iii)}] Let $Y \in M_n(F_q)$ such that $[Y] =  [E_{\{e_1+a_2e_2+ \cdots +a_n e_n\}}]$. Then $[\Psi(Y)] = [E_{\{e_1+\Upsilon(a_2)e_2+\cdots+\Upsilon(a_n)e_n\}}]$. Suppose that $Z \in M_n(F_q)$ such that $[Z] = [E_{\{e_1+e_{i}+a_{i+1}e_{i+1}+\cdots +a_n e_n\}}]$. It follows that $[\Psi(Z)] = [E_{\{e_1+e_i +\Upsilon(a_{i+1})e_{i+1}+ \cdots+\Upsilon(a_n)e_n\}}]$. Let $[\Psi(X)] = [E_{\{e_i+d_{i+1}e_{i+1}+ \cdots +d_ne_n\}}]$. By Lemma \ref{lemma17} and by using $[X]$ and $[Z]$, we have $d_j = \Upsilon(a_j)$, where $i < j \leq n$.
\end{enumerate}
\end{proof}
\begin{lemma}\label{lemma21}
Let $\Psi \in Aut(\overrightarrow{\Gamma_{l-i}}(M_n(F_q)))$ such that $[\Psi(E_{e_k})] = [E_{e_k}]$ where $1 \leq k \leq n$ and $[\Psi(E_\mathbf{1})] = [E_\mathbf{1}]$. Let $\Upsilon$ be a map as defined above. Then $\Upsilon$ is field automorphism of $F_q$.
\end{lemma}
\begin{proof}
Since $\Upsilon$ is bijective in $F_q$ so need to prove that 
\begin{center}
$\Upsilon(a+b) = \Upsilon(a)+\Upsilon(b)$ and,
    $\Upsilon(a\cdot b) = \Upsilon(a)\cdot \Upsilon(b)$ where $a,b \in F_q$.
\end{center}
If $a\cdot b = 0$ then either $a = 0$ or $b = 0$. It follows that either $\Upsilon(a) = 0$ or $\Upsilon(b) = 0$. Therefore, $\Upsilon(a)\cdot \Upsilon(b) = 0 = \Upsilon(a\cdot b)$. Hence, we assume that $a\cdot b \neq 0$ which implies that $\Upsilon(a)\cdot \Upsilon(b) \neq 0$. First we claim that $\Upsilon(a^{-1}) = \Upsilon(a)^{-1}$. Let $X_1, X_2 \in M_n(F_q)$ such that $[X_1] = [E_{\{e_1+ae_2+e_3\}}]$ and $[X_2] = [E_{\{e_2+a^{-1}e_3\}}]$. Then $[\Psi(X_1)] = [E_{\{e_1+\Upsilon(a)e_2+e_3\}}]$ and $[\Psi(X_2)] = [E_{\{e_2+\Upsilon(a^{-1})e_3\}}]$. Therefore, by Lemma \ref{lemma17} and by using $[X_1]$ and $[X_2]$, we have $\Upsilon(a)\cdot \Upsilon(a^{-1}) = 1$. It follows that $\Upsilon(a^{-1}) = \Upsilon(a)^{-1}$. Now suppose that $X_3, X_4 \in M_n(F_q)$ such that $[X_3] = [E_{\{e_1+abe_2+e_3\}}]$ and $[X_4] = [E_{\{e_1+be_2+a^{-1}e_3\}}]$. It follows that $[\Psi(X_3)] = [E_{\{e_1+\Upsilon(ab)e_2+e_3\}}] $ and $[\Psi(X_4)] = [E_{\{e_1+\Upsilon(b)e_2+\Upsilon(a^{-1})e_3\}}]$. In view of Lemma \ref{lemma17} and by using $[X_3]$ and $[X_4]$, we get $\Upsilon(ab)\cdot \Upsilon(a^{-1}) = \Upsilon(b)$. Consequently, $\Upsilon(ab) = \Upsilon(a)\cdot \Upsilon(b)$.

Let $Y_1, Y_2 \in M_n(F_q)$ such that $[Y_1] = [E_{\{e_1+ae_3\}}]$ and $[Y_2] = [E_{\{e_2+be_3\}}]$. Then $[\Psi(Y_1)] = [E_{\{e_1+\Upsilon(a)e_3\}}]$ and $[\Psi(Y_2)] = [E_{\{e_2+\Upsilon(b)e_3\}}]$. Now, $[Y_1] \subset  [E_{\{e_1+ae_3, e_2+be_3\}}]$ and $[Y_2] \subset  [E_{\{e_1+ae_3, e_2+be_3\}}]$ implies that $[E_{\{e_1+\Upsilon(a)e_3\}}] \subset [\Psi(E_{\{e_1+ae_3, e_2+be_3\}})]$ and $[E_{\{e_2+\Upsilon(b)e_3\}}] \subset [\Psi(E_{\{e_1+ae_3, e_2+be_3\}})]$. It follows that $[E_{\{e_1+\Upsilon(a)e_3\}}] + [E_{\{e_2+\Upsilon(b)e_3\}}] \subseteq [\Psi(E_{\{e_1+ae_3, e_2+be_3\}})]$. Since $\dim([E_{\{e_1+\Upsilon(a)e_3\}}] + [E_{\{e_2+\Upsilon(b)e_3\}}]) = \dim([\Psi(E_{\{e_1+ae_3, e_2+be_3\}})]) = 2n$ so $[E_{\{e_1+\Upsilon(a)e_3\}}] + [E_{\{e_2+\Upsilon(b)e_3\}}] = [\Psi(E_{\{e_1+ae_3, e_2+be_3\}})]$. Assume that $Y_3\in M_n(F_q)$ such that $[Y_3] = [E_{\{e_1+e_2+(a+b)e_3\}}]$. Then $[Y_3] \subset [E_{\{e_1+ae_3, e_2+be_3\}}]$ implies that $[\Psi(Y_3)] \subset [\Psi(E_{\{e_1+ae_3, e_2+be_3\}})]$. It follows that $[E_{\{e_1+e_2+\Upsilon(a+b)e_3\}}] \subset [\Psi(E_{\{e_1+ae_3, e_2+be_3\}})]$ so that $[E_{\{e_1+e_2+\Upsilon(a+b)e_3\}}] \subset [E_{\{e_1+\Upsilon(a)e_3\}}] + [E_{\{e_2+\Upsilon(b)e_3\}}]$. Consequently, $\Upsilon(a+b) = \Upsilon(a)+\Upsilon(b)$. Thus, $\Upsilon$ is field automorphism of $F_q$.
\end{proof}

Now, we extend $\Upsilon$ to $M_n(F_q)$ such that $\Upsilon(X) = [\Upsilon(x_{ij})]_{n\times n}$ for any $X  = [x_{ij}]_{n\times n} \in M_n(F_q)$. Then define a mapping on vertex set of $\overrightarrow{\Gamma_{l-i}}(M_n(F_q))$ and denote it by $\Upsilon$ only. From now onwards, whether we are referring to mapping on $F_q$ or mapping on $M_n(F_q)$ by $\Upsilon$, will be determined by the context. 
\begin{lemma}
The map $\Upsilon$ (as defined above) is an automorphism of $\overrightarrow{\Gamma_{l-i}}(M_n(F_q))$.
\end{lemma}
\begin{proof}
Let $X,Y \in M_n(F_q)$ such that $\Upsilon(X) = \Upsilon(Y)$. Then $[\Upsilon(x_{ij})]_{n\times n} = [\Upsilon(y_{ij})]_{n\times n}$, where $X = [x_{ij}]_{n\times n}$ and $Y = [y_{ij}]_{n\times n}$. Since $\Upsilon$ is bijective over $F_q$, we have $[x_{ij}]_{n\times n} = [y_{ij}]_{n\times n}$. Therefore, $X = Y$. Thus, $\Upsilon$ is one-one over $M_n(F_q)$ and because $M_n(F_q)$ is finite so $\Upsilon$ is bijective over $M_n(F_q)$.

Let $X \rightarrow Y$. Then $[X] \subset [Y]$. Further assume that $A \in [\Upsilon(X)]$. It follows that $A = C\Upsilon(X)$, where $C \in M_n(F_q)$. Let $U, V$ be arbitrary element of $M_n(F_q)$. Then $\Upsilon(U\cdot V) = [w_{ij}]_{n\times n} = [\Upsilon(\sum_{k=1}^nu_{ik}\cdot v_{kj})]_{n\times n}$,  where $U \cdot V = W \in M_n(F_q)$. Since $\Upsilon$ is field automorphism of $F_q$ we get $\Upsilon(U\cdot V) = [\sum_{k=1}^n\Upsilon(u_{ik})\cdot \Upsilon(v_{kj})]$ and  $[\sum_{k=1}^n\Upsilon(u_{ik})\cdot \Upsilon(v_{kj})] = \Upsilon(U)\cdot \Upsilon(V)$. Therefore, $\Upsilon(U\cdot V) = \Upsilon(U)\cdot \Upsilon(V)$.

Thus, $A = C\Upsilon(X) = \Upsilon(\Upsilon^{-1}(C)\cdot X)$ so that $\Upsilon^{-1}(A) = \Upsilon^{-1}(C)\cdot X$. It follows that $\Upsilon^{-1}(A) \in [X] \subset [Y]$. Therefore, there exists $D \in M_n(F_q)$ such that $\Upsilon^{-1}(A) = DY$ implies that $A = \Upsilon(D)\cdot \Upsilon(Y)$. Consequently, $A \in [\Upsilon(Y)]$ and so $[\Upsilon(X)] \subseteq [\Upsilon(Y)]$. Now, there exists $Z \in [Y]$ such that $Z \notin [X]$. Let $B \in M_n(F_q)$ such that $Z = BY$. Then $\Upsilon(Z) = \Upsilon(B)\cdot \Upsilon(Y)$. It implies that $\Upsilon(Z) \in [\Upsilon(Y)]$. If  $\Upsilon(Z) \in [\Upsilon(X)]$, then there exists $L \in M_n(F_q)$ such that $\Upsilon(Z) = L \cdot \Upsilon(X)$ and this gives $Z = \Upsilon^{-1}(L)\cdot X$ as $\Upsilon$ is bijective over $M_n(F_q)$. Therefore, $Z \in [X]$, a contradiction. Thus, $\Upsilon(Z) \in [\Upsilon(Y)]$ but $\Upsilon(Z) \notin [\Upsilon(X)]$. It follows that $[\Upsilon(X)] \subset [\Upsilon(Y)]$ and so $\Upsilon(X) \rightarrow \Upsilon(Y)$. Thus, $\Upsilon$ is automorphism of $\overrightarrow{\Gamma_{l-i}}(M_n(F_q))$.
\end{proof}
\begin{lemma}\label{lemma23}
If $[X] = [Y]$ then $[\Upsilon(X)] = [\Upsilon(Y)]$, where $X,Y \in M_n(F_q)$.
\end{lemma}
\begin{proof}
Let $A \in [\Upsilon(X)]$. Then there exists $W \in M_n(F_q)$ such that $A = W \cdot \Upsilon(X)$. It follows that $\Upsilon^{-1}(A) = \Upsilon^{-1}(W)\cdot X$ and so $\Upsilon^{-1}(A) \in [X] = [Y]$. Then there exists $C \in M_n(F_q)$ such that $\Upsilon^{-1}(A) = C \cdot Y$. Therefore, $A = \Upsilon(C) \cdot \Upsilon(Y)$ implies $A \in [\Upsilon(Y)]$. Thus, $[\Upsilon(X)] \subseteq [\Upsilon(Y)]$. Similarly, we can prove that $[\Upsilon(Y)] \subseteq [\Upsilon(X)]$. Hence, $[\Upsilon(X)] = [\Upsilon(Y)]$.
\end{proof}

\begin{lemma}\label{lemma24}
Let $\Psi \in Aut(\overrightarrow{\Gamma_{l-i}}(M_n(F_q)))$ such that $[\Psi(E_{e_k})] = [E_{e_k}]$, where $1 \leq k \leq n$ and $[\Psi(E_\mathbf{1})] = [E_\mathbf{1}]$ and let $\Upsilon$ be a map as defined above. Then for some $\Upsilon$, we have $[(\Upsilon^{-1}\cdot \Psi)(X)] = [X]$ and $X \in M_n(F_q)$.
\end{lemma}
\begin{proof}
In view of Corollary \ref{rank}, if $X = 0$ then $\Upsilon^{-1}\cdot \Psi(X) = X$. It follows that $[(\Upsilon^{-1}\cdot \Psi)(X)] = [X]$. Now, assume that $rank(X) = 1$. Then $[X] = [E_a]$ for some $0 \neq a \in F_q^n$, where $a = (e_i+a_{i+1}e_{i+1}+ \cdots +a_ne_n)$ and $1 \leq i \leq n$. By Lemma \ref{lemma20}, we obtain $[\Psi(X)] = [E_{\{e_i+\Upsilon(a_{i+1})e_{i+1}+ \cdots +\Upsilon(a_n)e_n\}}]$. By Lemma \ref{lemma23}, we have $[(\Upsilon^{-1}\cdot \Psi)(X)] = [\Upsilon^{-1}(E_{\{e_i+\Upsilon(a_{i+1})e_{i+1}+ \cdots +\Upsilon(a_n)e_n\}})]$. It follows that $[(\Upsilon^{-1}\cdot \Psi)(X)] = [E_{\{e_i+a_{i+1}e_{i+1}+ \cdots +a_ne_n\}}]$. Therefore, $[(\Upsilon^{-1}\cdot \Psi)(X)] = [X]$.

Let $rank(X) \geq 2$. Then by Lemma \ref{lemma4}, there exists $P \in M_n^*(F_q)$ such that $[X] = [E_rP]$, where $rank(X) = r$. Since $[E_{e_i}P] \subset [E_rP]$, where $1 \leq i \leq r$, we get $[(\Upsilon^{-1}\cdot \Psi)(E_{e_i}P)] \subset [(\Upsilon^{-1}\cdot \Psi)(X)]$. As $rank(E_{e_i}P) = 1$ implies that $[E_{e_i}P] \subset [(\Upsilon^{-1} \cdot \Psi)(X)]$. It follows that $\sum_{i =1}^r[E_{e_i}P] \subseteq [(\Upsilon^{-1}\cdot \Psi)(X)]$ and so $[E_rP] \subseteq [(\Upsilon^{-1}\cdot \Psi)(X)]$. Since, $\dim([E_rP]) = \dim([(\Upsilon^{-1}\cdot \Psi)(X)]) = nr$, we have  $[E_rP] = [(\Upsilon^{-1}\cdot \Psi)(X)]$. Therefore, $[X] = [(\Upsilon^{-1}\cdot \Psi)(X)]$ for $X \in M_n(F_q)$.
\end{proof}

Define a binary relation $\equiv$ on $M_n(F_q)$ such that $X \equiv Y$ if and only if $[X] = [Y]$. It is easy to check that $\equiv$ is an equivalence relation. Further, we define a mapping $\sigma$ over $M_n(F_q)$ such that it permute elements of each equivalence class arbitrarily.

\begin{lemma}\label{lemma25}
The map $\sigma$ is an automorphism of $\overrightarrow{\Gamma_{l-i}}(M_n(F_q))$.
\end{lemma}
\begin{proof}
By the definition, $\sigma$ is bijective over $M_n(F_q)$. Let $X, Y \in M_n(F_q)$ such that $X \rightarrow Y$. Then $[X] \subset [Y]$. Since $[\sigma(X)] = [X]$ and $[\sigma(Y)] = [Y]$ implies that $[\sigma(X)] \subset [\sigma(Y)]$. It follows that $\sigma(X) \rightarrow \sigma(Y)$. Thus, $\sigma$ is an automorphism of $\overrightarrow{\Gamma_{l-i}}(M_n(F_q))$.
\end{proof}
Now we prove our main result of this section.
\begin{theorem}\label{theorem26}
Let $n \geq 3$ and $\Psi \in Aut(\overrightarrow{\Gamma_{l-i}}(M_n(F_q)))$. Then there exists $P \in M_n^*(F_q)$, a field automorphism $\Upsilon$ and $\sigma$, as defined above, such that $\Psi = \varphi_{P} \cdot \Upsilon \cdot \sigma$.
\end{theorem}
\begin{proof}
By Lemma \ref{lemma16}, there exists $P \in M_n^*(F_q)$ such that $[(\varphi_{P^{-1}}\cdot \Psi)(E_{e_k})] = [E_{e_k}]$, where $1 \leq k \leq n$ and $[(\varphi_{P^{-1}}\cdot \Psi)(E_\mathbf{1})] = [E_\mathbf{1}]$. In view of Lemma \ref{lemma24}, there exists $\Upsilon$, a field automorphism of $F_q$ such that $[X] = [(\Upsilon^{-1}\cdot \varphi_{P^{-1}}\cdot \Psi)(X)]$ for all $X \in M_n(F_q)$. Thus, $\Upsilon^{-1}\cdot \varphi_{P^{-1}}\cdot \Psi = \sigma$. Hence, $\Psi = \varphi_P\cdot \Upsilon \cdot \sigma$.
\end{proof}
\begin{theorem}
Let $\Psi \in Aut(\overrightarrow{\Gamma_{l-i}}(M_2(F_q)))$. Then there exists $\rho$, as defined above, such that $\rho = \Psi$.
\end{theorem}
\begin{proof}
By Lemma \ref{lemma10}, $\rho$ is an automorphism of $\overrightarrow{\Gamma_{l-i}}(M_n(F_q))$. Since $rank(\Psi(X)) = rank(X)$, define $\rho$ over $M_2(F_q)$ as $\rho(X) = \Psi(X)$. It follows that $\rho^{-1}\cdot \Psi(X) = X$. Therefore, $\rho^{-1}\cdot \Psi = id$, where $id$ is identity mapping over $M_2(F_q)$. Thus, $\Psi = \rho$.
\end{proof}

\section{Graph Theoretic Properties of undirected Left-Ideal Relation Graph ${\Gamma_{l-i}}(M_n(F_q))$}

In this section, we consider the undirected left-ideal relation graph $\Gamma_{l-i}(M_n(F_q))$.  For $X, Y \in M_n(F_q)$ there is an edge between $X$ and $Y$ i.e. $X \sim Y$ if and only if $[X] \subset [Y]$ or $[Y] \subset [X]$. We investigate various graph theoretic properties of $\Gamma_{l-i}(M_n(F_q))$ including planarity, chromatic number, clique number and strong metric dimension.

\begin{theorem}
 The graph $\Gamma_{l-i}(M_n(F_q))$ is non-planar for $n \geq 2$.
 \end{theorem}
 \begin{proof}
 Let $e_{11} = (1, 1, 0, \cdots, 0) \in F_q^n$. Now, consider $V_1 = \{E_{e_1}, E_{e_2}, E_{e_{11}}\}$ and $V_2 = \{E_{\{e_1,e_2\}}, E_{\{e_2, e_1\}}, E_{\{e_{11},e_1\}}\}$. Observe that graph induced by $V_1 \cup V_2$ is isomorphic to $K_{3,3}$. Therefore, by Kurtowski's theorem, we have $\Gamma_{l-i}(M_n(F_q))$ is non-planar for $n \geq 2$.
  \end{proof}

\begin{theorem}\label{theorem4.3}
 For $M_n(F_q)$, we have  $\chi(\Gamma_{l-i}(M_n(F_q))) = \omega(\Gamma_{l-i}(M_n(F_q))) = n+1$.
 \end{theorem}
 \begin{proof}
 Consider $\mathcal{C} = \{E_0, E_1, E_2, \cdots, E_n\}$. Then the graph induced by $\mathcal{C}$ is clique of $\Gamma_{l-i}(M_n(F_q))$. It follows that $\omega(\Gamma_{l-i}(M_n(F_q))) \geq n+1$. Now we need to show that $\chi(\Gamma_{l-i}(M_n(F_q)))  \leq n+1$. Since no two matrices of same rank are adjacent in $\Gamma_{l-i}(M_n(F_q))$. Thus, we can assign one color to matrices of same rank. It follows that $\chi(\Gamma_{l-i}(M_n(F_q))) \leq n+1$. Thus, $\chi(\Gamma_{l-i}(M_n(F_q))) = \omega(\Gamma_{l-i}(M_n(F_q))) = n+1$.
  \end{proof}

\begin{lemma}
The eccentricity of $X \in V(\Gamma_{l-i}(M_n(F_q)))$ is given below:
\begin{equation*}
ecc(X) = 
\begin{cases}
1 & if \ X = 0\\
2 & if \ X \neq 0
\end{cases}
\end{equation*}
\end{lemma}
\begin{proof}
If $X = 0$ then $X$ is adjacent to every other vertex of $\Gamma_{l-i}(M_n(F_q))$. Hence, $ecc(X) = 1$ as eccentricity of a vertex is maximum distance of a vertex to any other vertex. Now, if $X \neq 0$ then $rank(X) \geq 1$. For $rank \geq 1$ there are atleast two vertices of that \(rank\). Let $Y \in M_n(F_q)$ such that $rank(X) = rank(Y)$. Then $d(X, Y) = 2$. Thus, the result holds.
\end{proof}
\begin{corollary}
The graph $\Gamma_{l-i}(M_n(F_q))$ is connected and $diam(\Gamma_{l-i}(M_n(F_q))) = 2$ and $r(\Gamma_{l-i}(M_n(F_q)) = 1$.
 \end{corollary}

\begin{corollary}
The dominance number of $\Gamma_{l-i}(M_n(F_q))$ is $1$.
 \end{corollary} 
\begin{theorem}
 If the graph $\Gamma_{l-i}(M_n(F_q))$ has a cycle, then $g(\Gamma_{l-i}(M_n(F_q))) = 3$.
 \end{theorem}
 \begin{proof}
 Suppose that $\Gamma_{l-i}(M_n(F_q))$ has a cycle. Then $n \geq 2$. Now 
 consider the set $C = \{E_0, E_1, E_n\}$. The graph induced by vertex set $C$ is a cycle of length $3$. Thus, $g(\Gamma_{l-i}(M_n(F_q))) = 3$.
 \end{proof}
\begin{theorem}
 The strong metric dimension of $\Gamma_{l-i}(M_n(F_q))$ is $q^{n^2}-n-1$.
 \end{theorem}
 \begin{proof}
 Since, $diam(\Gamma_{l-i}(M_n(F_q))) = 2$. Then by Theorem \ref{strong-metric-dim}, we have $ sdim(\Gamma_{l-i}(M_n(F_q))) = |M_n(F_q)|-\omega(R_{\Gamma_{l-i}(M_n(F_q))})$, where $R_{\Gamma_{l-i}(M_n(F_q))} $ is the reduced graph of $\Gamma_{l-i}(M_n(F_q))$. It is known that $|M_n(F_q)| = q^{n^2}$. Since  $\omega(\Gamma_{l-i}(M_n(F_q))) \geq \omega(R_{\Gamma_{l-i}(M_n(F_q))})$, by Theorem \ref{theorem4.3}, we have  $\omega(R_{\Gamma_{l-i}(M_n(F_q))}) \leq n+1$. Define a relation $\equiv$ on $M_n(F_q)$ such that $X \equiv Y$ if and only if $N(X) = N(Y)$, where $N(X) = \{Z\ |\ X \sim Z\ \text{and} \ Z\in M_n(F_q)\}\cup\{X\}$. Observe that $\equiv$ is an equivalence relation. Let $U(\Gamma_{l-i}(M_n(F_q)))$ be the complete set of representative elements for the above mentioned equivalence relation, then notice that $U(\Gamma_{l-i}(M_n(F_q)))$ is the vertex set of $R_{\Gamma_{l-i}(M_n(F_q))}$. Now we consider $\mathcal{C} = \{E_0, E_1, E_2, \cdots, E_n\} \subseteq U(\Gamma_{l-i}(M_n(F_q)))$. Note that $E_i$ and $E_j$ belong to distinct equivalence classes for $i\neq j$. As graph induced by $\mathcal{C}$ is a clique of $R_{\Gamma_{l-i}(M_n(F_q))}$. It follows that $\omega(R_{\Gamma_{l-i}(M_n(F_q))}) \geq n+1$. Therefore, $\omega(R_{\Gamma_{l-i}(M_n(F_q))}) = n+1$. Thus, $sdim(\Gamma_{l-i}(M_n(F_q))) = q^{n^2}-n-1$.
  \end{proof}
\begin{theorem}
The graph $\Gamma_{l-i}(M_n(F_q))$ is not Eulerian.
\end{theorem}
\begin{proof}
Let $deg(Y)$ denote degree of matrix $Y \in M_n(F_q)$ in $\Gamma_{l-i}(M_n(F_q))$. As, zero matrix is adjacent to every other matrix of $\Gamma_{l-i}(M_n(F_q))$. It follows that degree of zero matrix is $q^{n^2}-1$. If $q = 2^m$ for some $m \in \mathbb{N}$, then $q^{n^2}-1$ will be odd and hence degree of zero matrix is odd. 
If $X \in M_n(F_q)$ such that $rank(X) = n$ then $deg(X) = q^{n^2}-M(n,n,n,q)$ from \cite{abdel2012counting}, where $M(n,n,n,q)$ is number of $n\times n$ matrices of $rank\ n$ over a finite field $F_q$. Also, $M(n,n,n,q) = \prod_{j=0}^{n-1}(q^n-q^j)$. Now if $q = p^m$ for some odd prime $p$ and $m \in \mathbb{N}$, then $q^n-q^j$ will be even for all $0 \leq j \leq n-1$. It follows that $M(n,n,n,q)$ will be even and hence $q^{n^2}-M(n,n,n,q)$ will be odd as $q^{n^2}$ will be odd. Therefore, $deg(X)$ will be odd for $X \in M_n(F_q)\ and \ rank(X)=n$.
Thus, the zero matrix and any full rank matrix $X$ cannot have even degree simultaneously. Hence, $\Gamma_{l-i}(M_n(F_q))$ is not Eulerian.
\end{proof}

\section{Acknowledgement}
The first author wishes to acknowledge the support of MATRICS Grant  (MTR/2018/000779) funded by SERB, India. The second author gratefully acknowledge for providing financial support to CSIR  (09/719(0093)/2019-EMR-I) government of India.


\begin{thebibliography}{10}

\bibitem{abdel2012counting}
K.~A. Abdel-Ghaffar.
\newblock Counting matrices over finite fields having a given number of rows of
  unit weight.
\newblock {\em Linear algebra and its applications}, 436(7):2665--2669, 2012.

\bibitem{akbari2004commuting}
S.~Akbari, M.~Ghandehari, M.~Hadian, and A.~Mohammadian.
\newblock On commuting graphs of semisimple rings.
\newblock {\em Linear algebra and its applications}, 390:345--355, 2004.

\bibitem{akbari2015inclusion}
S.~Akbari, M.~Habibi, A.~Majidinya, and R.~Manaviyat.
\newblock The inclusion ideal graph of rings.
\newblock {\em Communications in Algebra}, 43(6):2457--2465, 2015.

\bibitem{akbari2006zero}
S.~Akbari and A.~Mohammadian.
\newblock Zero-divisor graphs of non-commutative rings.
\newblock {\em Journal of Algebra}, 296(2):462--479, 2006.

\bibitem{akbari2013some}
S.~Akbari, R.~Nikandish, and M.~Nikmehr.
\newblock Some results on the intersection graphs of ideals of rings.
\newblock {\em Journal of Algebra and its Applications}, 12(04):1250200, 2013.

\bibitem{anderson1999zero}
D.~F. Anderson and P.~S. Livingston.
\newblock The zero-divisor graph of a commutative ring.
\newblock {\em Journal of algebra}, 217(2):434--447, 1999.

\bibitem{beck1988coloring}
I.~Beck.
\newblock Coloring of commutative rings.
\newblock {\em Journal of algebra}, 116(1):208--226, 1988.

\bibitem{ma2018strong}
X.~Ma, M.~Feng, and K.~Wang.
\newblock The strong metric dimension of the power graph of a finite group.
\newblock {\em Discrete Applied Mathematics}, 239:159--164, 2018.

\bibitem{ma2016automorphisms}
X.~Ma, D.~Wang, and J.~Zhou.
\newblock Automorphisms of the zero-divisor graph over 2$\times$ 2 matrices.
\newblock {\em Journal of the Korean Mathematical Society}, 53(3):519--532,
  2016.

\bibitem{ma2016automorphism}
X.~Ma and D.~Wong.
\newblock Automorphism group of an ideal-relation graph over a matrix ring.
\newblock {\em Linear and Multilinear Algebra}, 64(2):309--320, 2016.

\bibitem{maimani2008comaximal}
H.~R. Maimani, M.~Salimi, A.~Sattari, and S.~Yassemi.
\newblock Comaximal graph of commutative rings.
\newblock {\em Journal of Algebra}, 319(4):1801--1808, 2008.

\bibitem{ou2020automorphism}
S.~Ou, D.~Wang, and F.~Tian.
\newblock The automorphism group of zero-divisor graph of a finite semisimple
  ring.
\newblock {\em Communications in Algebra}, 48(6):2388--2405, 2020.

\bibitem{redmond2002zero}
S.~P. Redmond.
\newblock The zero-divisor graph of a non-commutative ring.
\newblock {\em International J. Commutative Rings}, 1:203--211, 2002.

\bibitem{sharma1995note}
P.~K. Sharma and S.~Bhatwadekar.
\newblock A note on graphical representation of rings.
\newblock {\em Journal of Algebra}, 176(1):124--127, 1995.

\bibitem{wang2017automorphisms}
D.~Wang, L.~Chen, and F.~Tian.
\newblock Automorphisms of the co-maximal ideal graph over matrix ring.
\newblock {\em Journal of Algebra and Its Applications}, 16(12):1750226, 2017.

\bibitem{wang2009co}
H.-J. Wang.
\newblock Co-maximal graph of non-commutative rings.
\newblock {\em Linear algebra and its applications}, 430(2-3):633--641, 2009.

\bibitem{wang2016automorphisms}
L.~Wang.
\newblock Automorphisms of the zero-divisor graph of the ring of all $n\times
  n$ matrices over a finite field.
\newblock {\em Discrete Mathematics}, 339(8):2036--2041, 2016.

\bibitem{wang2020automorphisms}
L.~Wang, X.~Fang, and F.~Tian.
\newblock Automorphisms of the total graph over upper triangular matrices.
\newblock {\em Journal of Algebra and Its Applications}, 19(08):2050161, 2020.

\bibitem{westgraph}
D.~B. West et~al.
\newblock {\em Introduction to graph theory}, volume~2.
\newblock Prentice hall Upper Saddle River, 2001.

\bibitem{wong2014group}
D.~Wong, X.~Ma, and J.~Zhou.
\newblock The group of automorphisms of a zero-divisor graph based on rank one
  upper triangular matrices.
\newblock {\em Linear Algebra and its Applications}, 460:242--258, 2014.

\bibitem{xu2020automorphism}
F.~Xu, D.~Wong, and F.~Tian.
\newblock Automorphism group of the intersection graph of ideals over a matrix
  ring.
\newblock {\em Linear and Multilinear Algebra,
  https://doi.10.1080/03081087.2020.1723473}, 2020.

\bibitem{ye2012co}
M.~Ye and T.~Wu.
\newblock Co-maximal ideal graphs of commutative rings.
\newblock {\em Journal of Algebra and its Applications}, 11(06):1250114, 2012.

\bibitem{zhou2017automorphism}
J.~Zhou, D.~Wong, and X.~Ma.
\newblock Automorphism group of the total graph over a matrix ring.
\newblock {\em Linear and Multilinear Algebra}, 65(3):572--581, 2017.

\bibitem{zhou2017automorphisms}
J.~Zhou, D.~Wong, and X.~Ma.
\newblock Automorphisms of the zero-divisor graph of the full matrix ring.
\newblock {\em Linear and Multilinear Algebra}, 65(5):991--1002, 2017.

\end{thebibliography}
\end{document}